\documentclass[10pt]{amsart}   	% use "amsart" instead of "article" for AMSLaTeX format
\usepackage{geometry}                		% See geometry.pdf to learn the layout options. There are lots.
\usepackage[parfill]{parskip}    		% Activate to begin paragraphs with an empty line rather than an indent
\usepackage{graphicx}				% Use pdf, png, jpg, or eps§ with pdflatex; use eps in DVI mode
\usepackage{mathrsfs}  
					
\usepackage{amsfonts,amsthm}

\usepackage{amstext, xcolor} 

\usepackage{mathtools}
\usepackage{array}								
\usepackage{amssymb}

\usepackage{amsmath}

\usepackage{cite}

\newcommand*{\f}{A}%
\newcommand*{\g}{B}%
\newtheorem*{rep@theorem}{\rep@title}
\newcommand{\newreptheorem}[2]{%
\newenvironment{rep#1}[1]{%
 \def\rep@title{#2 \ref{##1}}%
 \begin{rep@theorem}}%
 {\end{rep@theorem}}}
\makeatother

\newtheorem{definition}{Definition}[section]
\newtheorem{theorem}{Theorem}[section]
\newtheorem{corollary}{Corollary}[theorem]
\newtheorem{lemma}[theorem]{Lemma}
\newtheorem{proposition}{Proposition}[theorem]
\newreptheorem{theorem}{Theorem}
\newtheorem*{remark}{Remark}

\begin{document}

\title[Quasi-isometric embeddings from $T_n$ to $T$]{% 
	Quasi-isometric embeddings from the generalised Thompson's group to Thompson's groups $T$
}
\author{% 
	Xiaobing Sheng
} 
\address{%  
	Graduate School of Mathematical Sciences, 
	The University of Tokyo
} 
\email{% 
	xiaobing@ms.u-tokyo.ac.jp 
} 


\date{\today} % delete this line to display the current date

\begin{abstract} 
Brown has defined the generalised Thompson's group  $F_n$,  $T_n$, $V_n$
where $n$ is an integer at least  $2$  and Thompson's groups  $F= F_2$, $T =T_2$, $V = V_2$  in the 80's. 
Burillo, Cleary and Stein have found that there is a quasi-isometric embedding from $F_n$ to $F_m$ 
where $n$ and $m$ are positive integers at least 2.  
We show that there is a quasi-isometric embedding from $T_n$ to $T_2$  for any  $n \geq 2$   
and no embeddings from  $T_2$  to  $T_n$  for  $n \geq 3$.  
%Discuss why it is not very likely to be the same in the case of $V_n$.
\end{abstract}

\maketitle

\section{Introduction}

Thompson's groups, $F, \, T$ and $V$
are some of the most mysterious groups being investigated in the last half century, 
and were first introduced by Richard Thompson 
for finding out finitely presented groups with unsolvable word problem. 
These groups were later found to have connections 
to string rewrite systems, combinatorics, homotopy theory and large scale geometry, etc \cite{JSMJ09,JMB04,VGMS97}. 
The groups can be considered 
as subgroups of the automorphism group of the Cantor set. 
There are purely algebraic interpretations of $F$ $T$ as finitely presented groups \cite{JWW},
\begin{equation} 
\label{pre1.1}
	F =\big< A, B \mid [AB^{-1}, A^{-1}BA], [AB^{-1}, A^{-2}BA^2] \big>  
\end{equation} 
where $[x,y] = xyx^{-1}y^{-1}$, 
and 
\begin{equation}
\label{pre1.2}
\begin{aligned} 
	T = \big< A, B, C\mid [AB^{-1}, A^{-1}BA], [AB^{-1}, A^{-2}BA^2], C^{-1}B(A^{-1}CB), \\
	((A^{-1}CB)(A^{-1}BA))^{-1}B(A^{-2}CB^{-2}), (CA)^{-1}(A^{-1}CB)^{2}, C^{3}
\big>.   
\end{aligned} 
\end{equation}

The groups have been generalised by Brown \cite{BKS87} in the late 80's 
to the families $F_n, T_n$ and $V_n$, where $n$ is an integer greater than or equal to $2$, 
which are the groups we are interested in, and where $F = F_2$, $T =T_2$ . 
Stein later on further generalised the groups into Thompson-Stein groups \cite{JWW,IL}. 
A comprehensive introduction is given in the introductory notes by Cannon, Floyd and Parry \cite{JWW}. 

Burillo \cite{J} first studied undistorted subgroups of $F$.
Burillo, Cleary and Stein \cite{JSM01} give a detailed construction of embeddings 
from $F_n$ into $F_m$ for any natural numbers $n, m \geq 2$.  
An estimate on the word length of the elements in the generalised Thompson's group $F_n$ is 
given through its unique normal form and it serves as an ingredient 
for proving that the embedding from $F_n$ into $F_m$ is quasi-isometric.

Liousse \cite{IL} has investigated the rotation numbers for Thompson-Stein groups 
and raised the question: 
Is it possible to describe all Thompson-Stein groups up to isomorphism, 
up to quasi-isomorphism?
She proves that selected Thompson-Stein groups are not isomorphic to each other. 

In this paper, we focus on Thompson's group $T_2$ and their generalisation $T_n$ \cite{BKS87}.  
We prove that there is a quasi-isometric embedding from $T_n $ to $T_2 $ for 
any integer $n \geq 2$ and there is no embedding from $T_2 $ to $T_n $ where $n \geq 3$.
 
%% V_n

%%%
%%%
%%%
\section{Background}
%%%
%%%
\subsection{Thompson's group  $F$  and its generalisations $F_n$}

We inherit the interpretations from  \cite{JMB04} and \cite{JWW}. 
Consider the interval  $[0, 1]$ with finitely many distinguished points at dyadic rationals 
and we have the following.

\begin{definition}
Let the interval  $[0,1]$  be divided so that the subintervals are only 
in the form  $\big[ \frac{\ell}{2^k}, \frac{\ell+1}{2^k}\big]$  where  $\ell+1< 2^k$  
and  $k,\ell \in \mathbb{N}\cup \{0\}$.  
When the interval  $[0,1]$  is divided in such a way and the number of subintervals are finite, 
then the division is called a dyadic subdivision and the interval  $[0,1]$  with a dyadic subdivision 
is called a dyadically subdivided interval.
\end{definition}

\begin{definition}[\cite{JWW}]  
$F$  is defined as the group of orientation-preserving 
piecewise-linear homeomorphisms from the unit interval  $[0,1]$ to itself 
which map dyadically subdivided intervals to dyadically subdivided intervals. 
\end{definition}

%The homeomorphisms defined are in fact differentiable except 
%at finitely many dyadic rational numbers and their derivatives are 
%powers of  $2$  on intervals of differentiability.

The following piecewise-linear functions are given as a finite list of generators of  
$F$ % \cite{JWW}.%{

\begin{gather*}
\f(t)=\left\{\begin{array}{lll}
               \frac{1}{2}t  & \textnormal{if} &  0\leq t\leq \frac{1}{2} \\
                t-\frac{1}{4}  &  \textnormal{if} & \frac{1}{2} \leq t \leq \frac{3}{4} \\ 
                              2t -1 & \textnormal{if} &\frac{3}{4} \leq t \leq 1
            \end{array}\right.
\quad
\g(t)=\left\{\begin{array}{lll}
            t& \text{if }  0\leq t\leq \frac{1}{2} \\
             \frac{1}{2}t +\frac{1}{4} & \textnormal{if}  \frac{1}{2} \leq t \leq \frac{3}{4} \\ 
                            t-\frac{1}{8} & \textnormal{if}  \frac{3}{4} \leq t \leq \frac{7}{8} \\
                2t -1 & \textnormal{if} \frac{7}{8} \leq t \leq 1
            \end{array}\right. %\label{eqB}
\end{gather*}

A generalised version of $F$ can be defined as follows,

\begin{definition}\label{def1}  
$F_n$  is defined as the group of orientation-preserving piecewise-linear homeomorphisms 
from the n-adically subdivided intervals to the n-adically subdivided intervals.
\end{definition}

\subsubsection{An alternative view on $F$}

To see the elements of  $F$  from another point of view, 
we first consider binary trees.  
Define a 2-caret to be a vertex with two edges attached to it.%(Figure \ref{fig1.5}).  
%We can see that 2-carets are building blocks of binary trees.

%
%
%

Then, we can relate a dyadically subdivided interval $[0,1]$ to a binary tree 
by associating a subinterval with a breakpoint in the middle to a 2-caret. 
A binary tree can be expanded simply by attaching 2-carets at some leaves and reduced by deleting 2-carets. 

For two binary trees with the same number of leaves, 
we label the leaves of each tree from the left to right by  $0, 1, 2, 3, \cdots, n$, 
and denote these two trees by  $\mathcal{U}$  and  $\mathcal{V}$,  
respectively, and denote the pair by  $(\mathcal{U}, \mathcal{V})$.  
All such pairs form a set  $\mathscr{T}$, 
i.e.  
$$\mathscr{T} = \{(\mathcal{U}, \mathcal{V}) \mid  \mathcal{U}, \mathcal{V} \; 
\mbox{are trees which have the same number of leaves}  \}.$$  
A tree pair as described above can be associated with some piecewise-linear functions 
mapping $[0,1]$  to itself and we call the former tree the ``source tree" and the later tree the ``target tree".

%Next, we introduce an equivalence relation on the set  $\mathscr{T}$.  
For a tree pair  $(\mathcal{U}, \mathcal{V})$, 
when we attach the roots of two 2-carets to each of the leaves on  
$\mathcal{U}$  and  $\mathcal{V}$  with the same numeric labeling, 
the corresponding piecewise-linear function is preserved and 
we call this procedure ``simple expansion".  
Similarly, 
when we delete two 2-carets with the same labeling at its leaves on each tree,  
the resulting tree pair still represents the same piecewise-linear function, 
``simple contraction".

Note that ``simple expansion" of a tree pair here is different from the expansion of a single tree that we described above.
A tree pair is reduced, if one can apply no more simple contractions on it.

We define an equivalence relation  $\sim$  on the set of tree pairs $\mathscr{T}$ as follows. 
Let  $(\mathcal{T}, \mathcal{S})$  and  $(\mathcal{T}', \mathcal{S}') \in \mathscr{T}$,  
we say that  $(\mathcal{T}, \mathcal{S}) \sim (\mathcal{T}', \mathcal{S}')$,  
if one can be obtained from the other by applying finite number of simple 
expansions and contractions.  
We hence denote the equivalence class of  
$(\mathcal{T}, \mathcal{S})$  by  $[(\mathcal{T}, \mathcal{S})]$.  
It is proved in \cite{JSM01} that every element of  $F$  is represented 
by a unique reduced tree pair.

%%%

%%%
\subsubsection{Binary operation}

Let  $(\mathcal{T}, \mathcal{S})$  and  $(\mathcal{U}, \mathcal{V})$  be 
two tree pairs in  $\mathscr{T}$.  
By applying simple expansions (or simple contractions) to $(\mathcal{T}$, $\mathcal{S})$  and  
$(\mathcal{U}$, $\mathcal{V})$ so that the trees  $\mathcal{S}$  and  $\mathcal{U}$  
become identical trees, 
say $\mathcal{R}$, 
we obtain the product of the two tree pairs, the pair  $(\mathcal{W}, \mathcal{Y})$, 
as the result.

$(\mathcal{W}, \mathcal{Y})$  is again a tree pair with the same number of leaves.  
This operation induces the multiplication  $\ast$  on the equivalence class  $\mathscr{T}/\sim$. 
One can check that  $\mathscr{T}/\sim$  with  $\ast$  forms a group.

Similarly, 
the elements in  $F_n$  can be represented as equivalence classes of 
some $n$-ary tree pairs whose source and target tree have the same number of leaves.

\subsection{Thompson's group $T$ and its generalisation $T_n$}

$T$,  
known as a subgroup of the group homeomorphisms of the unit interval 
with two end points identified  $([0, 1] /\sim \, = S^1)$  is defined as follows.

\begin{definition} 
We call a dyadically subdivided interval  $[0,1]$  with the end point identified, 
a dyadically subdivided circle $S^1$, or a circle with dyadic subdivision.
\end{definition}

%Note: by our earlier assumptions these subdivisions are divided into finitely many subintervals.

\begin{definition}
$T$  is defined as the group of orientation-preserving piecewise-linear homeomorphisms 
of $S^1$ which take a dyadically subdivided  $S^1$  to a dyadically subdivided  $S^1$ using affine maps on each subinterval.
\end{definition} 
The element expressed as a piecewise-linear function below is a torsion element in $T$.  
Together with  $A(t)$  and  $B(t)$  in  $F$, (Section 2),  
they generate $T$ \cite{JWW}.

{
 \abovedisplayskip=0pt\relax
\[
C(t) =
\left\{
\!
\begin{array}{lr}
	t+\frac{1}{2} & \text{ if }  0\leq t\leq \frac{1}{2} \\
	&\\
	t - \frac{1}{2} & \text{ if }  \frac{1}{2} \leq t \leq 1.   
\end{array}
\right.
\]}

Since the elements in  $T$  are homeomorphisms of  $S^1$  rather than the unit interval, 
we express these elements as binary tree pairs with ``cyclic" labeling.

We explain what ``cyclic" labeling means below.  
For a binary tree we pick some leaf and label it by  $0$, 
then the leaves following it on the right are labeled by non negative integers  $ 1, \cdots, k$  in order 
until the rightmost one.  
If all the leaves are labelled, 
then we are done.  
The label is exactly as in a tree of the tree pairs representing  $F$.  
If there are still leaves which are not labelled, 
these leaves are the ones on the left of the leaf labelled  $0$. 
We label these leaves from the leftmost by  $k+1$  onwards until reaching the leaf 
on the left of the leaf labeled by  $0$. 

Consider the set of binary tree pairs, the binary tree pairs  $(\mathcal{U}, \mathcal{V})$  
being labelled cyclically forms a set.  
By defining the equivalence relation and binary operation as in the case of $F$,  
we obtain a group isomorphic to $T$.

Every element of  $T$  as a tree pair also has a unique reduced tree pair.  
The uniqueness of tree pair representing elements in  $F$  first proved 
in \cite{JWW} and the analogue in  $T$  is explained in details in \cite{JSMJ09}. 

Analogously, the generalised groups  $T_n$ can be defined as follows, 
\begin{definition}
$T_n$  is defined as the group of orientation-preserving piecewise-linear homeomorphisms 
of  $S^1$  which take an $n$-adically subdivided $S^1$ to an $n$-adically subdivided $S^1$  
using affine maps on each subinterval.
\end{definition}

In the case of  $T_n$, 
we have elements which map subintervals on the $n$-adic subdivision of  $S^1$  to 
subintervals on the other $n$-adic subdivision of  $S^1$  with the same number of breakpoints, 
hence, the trees in tree pairs are $n$-ary trees. 
The uniqueness of the reduced tree pairs in  $T_n$  can be proved as in the case of  $T$.

%%%
%%%

%%%
%%%
\subsection{Known results}

In \cite{JSM01}, 
Burillo, 
Cleary and Stein have proved that for  $F_p$  and  $F_q$  
where  $p$  and  $q$ are positive integers at least $2$, 
there is a quasi-isometric embedding from one to the other defined by 
the ``caret-replacement map''.  
Here $q$-carets in $q$-ary tree pairs representing elements in  $F_q$  are 
replaced by a possible arrangement of $p$-carets 
to form $p$-ary tree pairs representing elements in  $F_p$. 
For the later argument, 
we give a proof of the special case which resembles the one of \cite[Subsection 5.3]{GGMS2}.  

\begin{lemma}
There is an embedding from $F_n$ to $F_2 = F.$
\label{lemma1}
\end{lemma}

\begin{proof}\label{proof1} 
Define a map  $\psi : \{\text{$n$-ary tree pairs}\} \rightarrow \{\text{binary tree pairs}\}$  as follows. 
For every element in  $F_n$  represented as an  $n$-ary tree pair, 
we replace every $n$-caret by an all-right binary tree 
(all carets except for the top caret are attached to the right most leaf of the previous caret)  
with  $n$  leaves.  
By this replacement, 
we can always obtain a binary tree pair from an $n$-ary one. 

To each simple expansion on a  $n$-ary tree pair  $(\mathcal{U}, \mathcal{V})$, 
we apply simple expansions  $n-1$  times  
on  $\psi((\mathcal{U}, \mathcal{V}))$  
by adding an all-right binary tree with  $n$  leaves  to the corresponding leaf.  
To each simple contraction, 
we apply the converse operation. 
These expanding or contracting operations obviously commute with the map  $\psi$, 
and hence,  
$\psi$  induces a group homomorphism  $\psi_* : F_n \to F_2$.   

It remains to show that $\psi_{\ast}$ is injective. We consider a non identity element in $F_n$ represented 
by a reduced  $n$-ary tree pair which consists of two different trees. 
The image of this tree pair by  $\psi$  representing an element in $F_2$ must consist of two different trees. 
Therefore, $\psi_{\ast}$ is injective. 
\end{proof}
%%%
%%%
%%%
\section{Quasi-isometric embeddings}

%%%
%%%
\subsection{Elements in $T_n$ represented as tree pairs}

From this subsection on, we will be focusing on the tree pair representation of the elements in $T_n$.  
When we talk about number of leaves or the number carets, we are talking about the number of those 
in either tree of a tree pair.

%%%
\subsubsection{Torsion elements in $T_n$} 

$T_n$  obviously contains torsion elements as  $T_2$  does, 
and Burillo, Cleary, Stein and Taback found 
a convenient tree pair representative of a torsion element as follows.   

\begin{theorem}[{\cite[Prop.6.1]{JSMJ09}}]
If $f \in T_2$ is a torsion element, 
then it can be represented by a (labelled) tree pair with the same source and target trees. 
\label{thm2}
\end{theorem}

\begin{proposition}\label{prop311}
If $f \in T_n$ is a torsion element, then it can be represented by 
a (labelled) tree pair with the same source and target trees.
\end{proposition}

\begin{proof}
Let  $m$  be the order of  $f$.  
Following the notation in \cite{JSMJ09}, 
we denote $f = (\mathcal{S}, \mathcal{T})$ for simplicity 
and we write $f =(\mathcal{S}, f(\mathcal{S}))$  for the convenience of the later argument. 
Expansion here simply means the expansion of a single tree.

We first construct the tree pair for  $f^k = (\mathcal{S}_k, \mathcal{T}_k)$  for  $k < m$ 
and then we could prove our hypothesis: $\mathcal{T}_{k+1}$  is an expansion of $\mathcal{T}_k$, by induction. 

\begin{itemize} 
\item 
	The base case:  we have  
	$f^2 = (\mathcal{S}, \mathcal{T})(\mathcal{S}, \mathcal{T}) 
	= (f^{-1}(\mathcal{E}_1), \mathcal{E}_1)(\mathcal{E}_1, f(\mathcal{E}_1))$, 
	where  $\mathcal{E}_1$  is the minimal joint expansion of  $\mathcal{T}$  and  $\mathcal{S}$.  
	Then define  $f^2 = (\mathcal{S}_2, \mathcal{T}_2)$  to be  $(f^{-1}(\mathcal{E}_1), f(\mathcal{E}_1))$.    
	Then  $\mathcal{T}_2 = f(\mathcal{E}_1) \supset \mathcal{T}$.   
\item 
	The inductive step: we assume the hypothesis is true for the case $k < m$. 
	We have the following, 
	\begin{equation*}
		(\mathcal{S}_{k-1}, \mathcal{T}_{k-1}) (\mathcal{S} , \mathcal{T}) 
		= (f^{-(k-1)}(\mathcal{E}_{k-1}), \mathcal{E}_{k-1})(\mathcal{E}_{k-1},f(\mathcal{E}_{k-1})) 
		= (\mathcal{S}_{k}, \mathcal{T}_{k}), 
	\end{equation*} 
	where  $\mathcal{E}_{k-1}$  is the minimal joint expansion 
	 $\mathcal{T}_{k-1}$  and  $\mathcal{S}$. 
	\begin{equation*}
		(\mathcal{S}_{k}, \mathcal{T}_{k}) (\mathcal{S} , \mathcal{T}) 
		= (f^{-k}(\mathcal{E}_{k}), \mathcal{E}_{k})(\mathcal{E}_{k},f(\mathcal{E}_{k})) 
		= (\mathcal{S}_{k+1}, \mathcal{T}_{k+1})
	\end{equation*} 
	where  $\mathcal{E}_{k}$  is the minimal  expansion of   
	$\mathcal{T}_{k}$  and  $\mathcal{S}$.
	Since $\mathcal{E}_{k}$  is an expansion of  both  $\mathcal{T}_k$  and  $\mathcal{S}$, 
	and in particular so does  $\mathcal{T}_{k-1}$  by the induction hypothesis.  
	Also since  $\mathcal{E}_{k -1}$  is a minimal expansion of  
	$\mathcal{T}_{k-1}$  and  $\mathcal{S}$,  
	$\mathcal{E}_k$  is an expansion of  $\mathcal{E}_{k-1}$.     
	Thus, 
	$\mathcal{T}_{k+1} = f(\mathcal{E}_{k}) \supset f(\mathcal{E}_{k-1}) = \mathcal{T}_{k}$  
\end{itemize}

By this inductive construction, 
we have  $id = f^m = (\mathcal{S}_m, \mathcal{T}_m)$, 
and thus  $\mathcal{S}_m = \mathcal{T}_m$.   
$(\mathcal{S}_m, \mathcal{T}_m)$ is represented by 
$$
(\mathcal{S}_m, \mathcal{T}_m) 
= (f^{-(m-1)}(\mathcal{E}_{m-1}), \mathcal{E}_{m-1})(\mathcal{E}_{m-1}, f(\mathcal{E}_{m-1})) 
= (\mathcal{S}_{m-1}, \mathcal{T}_{m-1})(\mathcal{S}, \mathcal{T}).   
$$ 

Since  $\mathcal{T}_m = f(\mathcal{E}_{m-1})$  is an expansion of  $\mathcal{T}_{m-1}$,   
and $\mathcal{E}_{m-1}$  is the minimal joint expansion of 
both  $\mathcal{T}_{m-1}$  and  $\mathcal{S}$,  
$f(\mathcal{E}_{m-1})$  is an expansion  $\mathcal{E}_{m-1}$.  
Also since  $f(\mathcal{E}_{m-1})$  and  $\mathcal{E}_{m-1}$  are 
the target and source trees of  $f^{m-1}$,  
which indicates that they have the same number of carets.  
Hence $\mathcal{E}_{m-1} = f(\mathcal{E}_{m-1})$ and they the identical trees in one tree pair.  
\end{proof}

%%%%%%%%%%%%

Having the above result, 
we could obtain upper bounds for the order of torsion elements in $T_n$ simply by counting the number leaves of a tree.
The following proposition we obtain is a result of this argument with \cite[Theorem 1]{IL}.
  
\begin{proposition}\label{prop312} 
The order of a torsion element in $T_n$ is a divisor of $l(n-1)+1$ for some integer $l \geq 0$.
\end{proposition}

\begin{proof} 
Following Proposition \ref{prop311}, 
we know that a torsion element always shifts the labeling of leaves.  
Let a tree pair  $(\mathcal{T}, \mathcal{T})$  represent a torsion element in  $T_n$.  
The number of leaves in  $\mathcal{T}$  is  $l(n-1)+1$,  
where $l$ is the number of carets.  
Hence the order of an element is a divisor of the number of leaves. 
\end{proof}

%The torsion elements in $V_n \backslash T_n$ is slightly more complicated.
%They coincides with non cyclic torsion elements in the symmetric group $\mathcal{S}_m$. 

%%%
\subsubsection{The tree pair representatives} 

Recall the presentation of  $F_2$  in (\ref{pre1.1}), 
define $x_0 = A(t)^{-1}$, $x_1 = B(t)^{-1}$ and 
\begin{equation*}  
	x_{i} = x_0^{-1}x_{i-1}x_0 
\end{equation*} 
for integer  $i \geq 2$  recursively. 
By the definition above, $x_i$'s are also maps of dyadically subdivided unit intervals and 
can be represented as binary tree pairs. 
%as in Figure \ref{fig10.2}.   

These tree pairs satisfy the infinite presentation 
of $F_2$ as follows (see for instance \cite{JWW}),  
\begin{equation}\label{pre3.1}
	F_2 = \big< x_i, i\geq 0 \mid x_i^{-1}x_jx_i = x_{j+1}  \mbox{ for }  i < j \big> .
\end{equation}
$F_n$  has a quite similar infinite presentation.  
Start with a finite set  $\{x_0, x_1, \cdots, x_{n-1}\}$  represented 
by  $n$-ary tree pairs in Figure \ref{fig8.7} and define 
\begin{equation*} 
	x_{i} = x_0^{-1} x_{i-(n-1)} x_0 
\end{equation*}  
for integer $i > n-1$  recursively.  
Then, 
it is shown in  \cite{JSM01}  that $F_n$  admits an infinite presentation as follows, 
\begin{equation}\label{pre3.2}
F_n = \big<  x_i, i \geq 0 \mid x_i^{-1} x_j x_i = x_{j+n-1} \textit{ for } i < j \big>. 
\end{equation} 

Let us denote the finite generating set  $\{ x_i \in F_n \, | \, 0 \leq i < n \}$  by  
$\Sigma_n$  and the infinite generating set  $\{ x_i \in F_n \, | \, i \geq 0 \}$  by  
$\Sigma_n'$.    

\begin{figure}[h!]
\centerline{\includegraphics[width=3.5in]{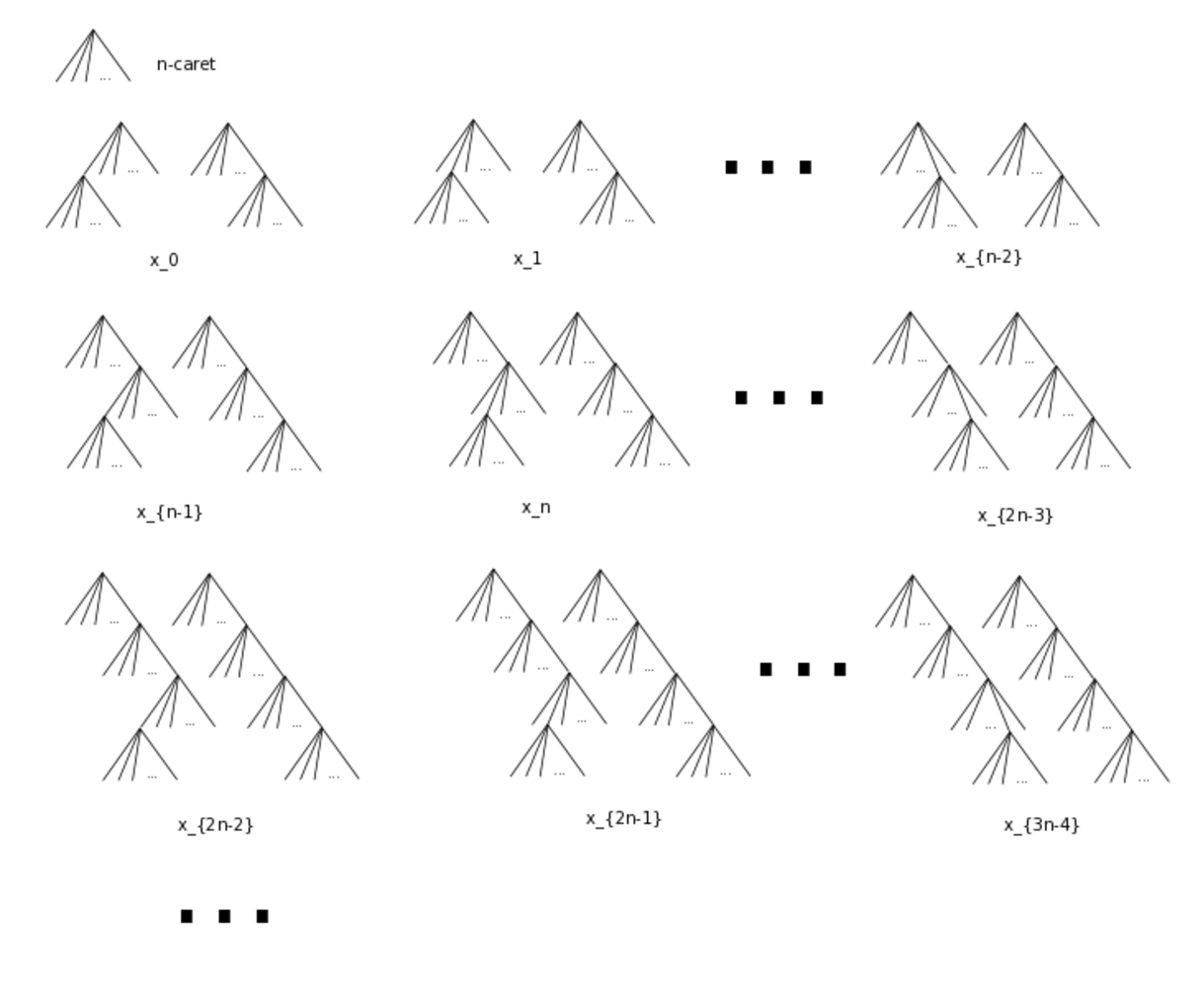}}
\vspace*{8pt}
\caption{An infinite generating set of $F_n$ in the form of tree pairs 
\label{fig8.7}}
\end{figure}

An infinite family of torsion elements
denoted by $c_{k-1}$   for positive integers  $k$  in  $T_n$  can be defined as 
a pair of all right trees with  $k$  carets  in  $T_n$  (Figure \ref{fig8.8}), 
where  $c_{k-1}$  is labeled so that the labelling the target tree is shifted 
to the left by one the source tree. 
$c_1$ can be identified with $C(t)$.  

\begin{figure}[h!]
\centerline{\includegraphics[width=3.5in]{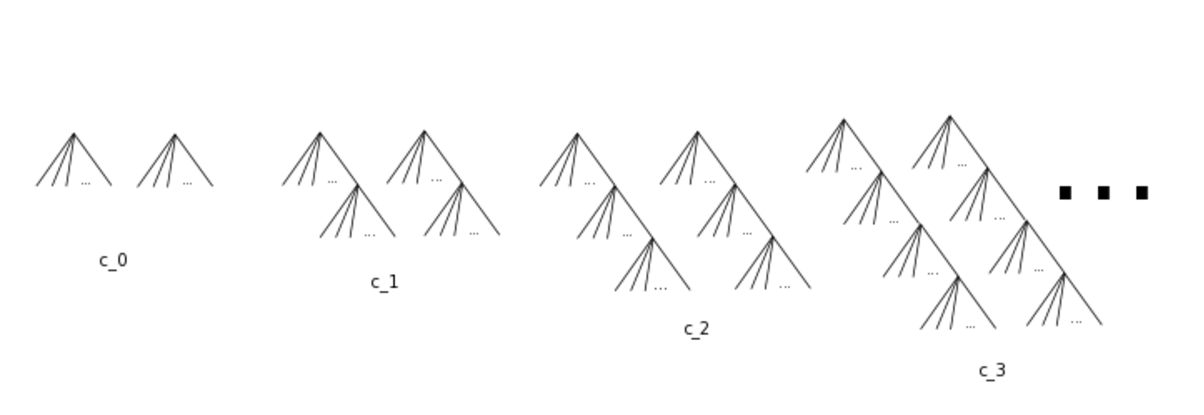}}
\vspace*{8pt}
\caption{An infinite series of torsion elements in $T_n$ in the form of tree pairs
\label{fig8.8}}
\end{figure}

It is known by  \cite{JSMJ09}  that 
\begin{align*} 
	\Sigma & = \{x_0, x_1, \ldots, x_{n-1}, c_0\},   \\ 
	\Sigma'&= \Sigma \cup \{c_k \, | \, \  1 \leq k \}  \\
	 & = \{x_0, x_1, \ldots, x_{n-1}, c_0, c_1, \ldots   \},
\end{align*}  

both form a generating set of  $T_n$,   
We will use both for the later argument.    

A finite presentation of  $T_2$  is given in \cite{JWW}, 
from which Burillo, Cleary, Stein and Taback  \cite{JSMJ09}  
have deduced what they call the pumping lemma by which they can  
reduce  $c_k$  by a word in  $c_{k-1}$  and  $x_i$'s.  
We have found an analogue of pumping relations for $T_n$ 
through the computation of the tree pairs representing elements in $T_n$.

\begin{lemma}\label{lm1} 
If  $\ell < (k-1)(n-1) + 1 = {\rm ord} c_{k-1}$,  
then 
\begin{align*} 
	\; c_{k}^{\ell} & = c_{k-1}^{\ell}x_{k(n-1)-\ell}, \quad \text{and} \\ 
	\; (c_k^{-1})^{\ell} & = x_{k(n-1)-\ell}^{-1} (c_{k-1}^{-1})^{\ell}.   
\end{align*} 
\end{lemma}

\begin{proof}
Both identities are proved through computation using standard computation.
%The computation shown as in Figure \ref{fig9.1}  proves the first identity.  
%The second one is just the inverse of the first identity (Figure \ref{fig9.2}).    
 \end{proof}
%
%

%%%
%%%

%%%
%%%
\subsection{A construction of an embedding from $T_n $ to $T_2 $}
\begin{theorem}\label{thm3.7}
There is an embedding from $T_n$ to $T_2= T$.
\end{theorem}

\begin{proof} 
Define a map  $\phi : \{\text{labelled $n$-ary tree pairs}\} \rightarrow \{\text{labelled binary tree pairs}\}$  
as follows. 
For every element in  $T_n$  represented as an $n$-ary tree pair, 
we replace every $n$-caret by an all-right binary tree with  $n$  leaves.  
Hence, we can always obtain a binary tree pair from an $n$-ary one. 
$\phi$ induces a group homomorphism  $\phi_* : T_n \to T_2 $ by the fact that simple expansion (contraction) is compatible with the operations in groups.   
Injectivity is proved similar as in the case of $F_n \rightarrow F$.  
\end{proof} 
\begin{remark}
This argument also works for the generalised group in the Thompson's group $V$ case, however, we are not going to discuss here.
\end{remark}

Despite of the embedding result of  $F_2$  to  $F_n$  in \cite{JSM01}, 
we cannot go other way around.  

\begin{theorem}
There is no injective homomorphism from $T_2 $ to $T_n $ where $n > 2$.
\label{thm1}
\end{theorem}

\begin{proof} 
By Proposition \ref{prop312}, as well as \cite[Consequence 3 in Theorem1]{IL}
we know that $T_n$ does not contain torsion elements of order $n-1$.  
However, $T_2$ does contain torsion elements of order $n-1$. 
\end{proof}

\begin{corollary}
There is no injective homomorphism from $T_{p}$ to $T_{q}$ when $p, q$ are two consecutive numbers.  
\end{corollary}

%%%
%%%
\subsection{An estimate of the word length of elements in $T_n$}

\subsubsection{Word length of elements in $T_n$}
For estimating the word length of elements in $T_n$, 
we extend the argument in \cite{JSMJ09}. 

Extending the definition in \cite[Page 9]{JSMJ09},  
we start with the unique reduced tree pair representing an element in  $T_n $   and 
represent it by the concatenation of three words in  $\Sigma'$.  
The formal definition is in the following, 

\begin{definition}[$\textbf{pcq}$ factorization]
Let the reduced labelled tree pair  $(\mathcal{T}_{-}, \mathcal{T}_{+} )$  represent 
an element in  $T_n(V_n)$  and  $\mathcal{T}_{-}$  and  $\mathcal{T}_{+}$  each have  
$k$ carets.  
Let $\mathcal{R}$ be the all right tree with $k$ carets.  
We write the element as a product $\textbf{p}\textbf{c}\textbf{q}$,  
where:
\begin{enumerate}
\item 
	$\textbf{p}$,   
	a positive word in the infinite generating set  $\Sigma_n$  of  $F_{n}$  with 
	the form $\textbf{p} = x_{i_1}^{r_1}x_{i_2}^{r_2} \cdots x_{i_y}^{r_y}$  
	where  $i_1 < i_2 < \cdots < i_y $  and all  $r_k$'s  are positive.  
	It is represented by  $(\mathcal{R}, \mathcal{T}_{+})$. 
%\item 
       %\begin{itemize}
       
	\item  For $\omega$ representing an element in $T_n$,
	$\sigma = c_{k}^{\ell}$, 
	where ${\ell}$  appearing in the exponent of  $c_k$  satisfies 
	$0 \leq \ell < k(n-1)+1 = {\rm ord} \, c_{k}$.  
	This element can be represented by  $(\mathcal{R}, \mathcal{R})$  
	with appropriate labeling. 
	
	%\item For $\omega$ representing an element in $V_n$,
	%$\sigma \in  \{\pi_{k(n-1)+1, m}^l,   0 \leq k, 0 \leq m \leq n\}$, 
	%where $0 \leq l < {\rm ord} \, S_{k(n-1)+1} $.  
	%This element can be represented by  $(\mathcal{R}, \mathcal{R})$  
	%with labeling of permutations in $S_{k(n-1)+1}$. 
	
	%\end{itemize}
\item 
	$\textbf{q}$, 
	a negative word in  $\Sigma_n$  of  $F_n$  with 
	the form  $x_{j_z}^{-s_z} \cdots x_{j_2}^{-s_2}x_{j_1}^{-s_1}$  
	where $j_1< j_2 <\cdots < j_z$  and all  $s_k$'s  are positive.  
	It is represented by $(\mathcal{T}_{-}, \mathcal{R})$.   
\end{enumerate} 
We call such a product a {\em $\textbf{pcq}$ factorization}. 
\label{def3.7}
\end{definition}

A  $\textbf{pcq}$  factorization for an element in  $T_n$  can be found 
by seeking a obvious path from  $\mathcal{T}_{\pm}$  to  $\mathcal{R}$  
starting from the reduced pair  $(\mathcal{T}_{-}, \mathcal{T}_{+})$ (Figure \ref{fig10.1}).     
Notice that the tree pairs representing $\textbf{p}$ and $\textbf{q}$ are not necessarily reduced 
though   $(\mathcal{T}_{-}, \mathcal{T}_{+})$  is. 
Moreover, 
when the tree pair representing some element in  $F_n $  is reduced, 
the torsion part of the $\textbf{pcq}$ factorization represents just an identity element.
%when the tree pair representing some element in $T_n \backslash V_n$, 
%the torsion part is just some cyclic torsion element.

%
%
%
\begin{figure}[h]
\centerline{\includegraphics[width=3.5in]{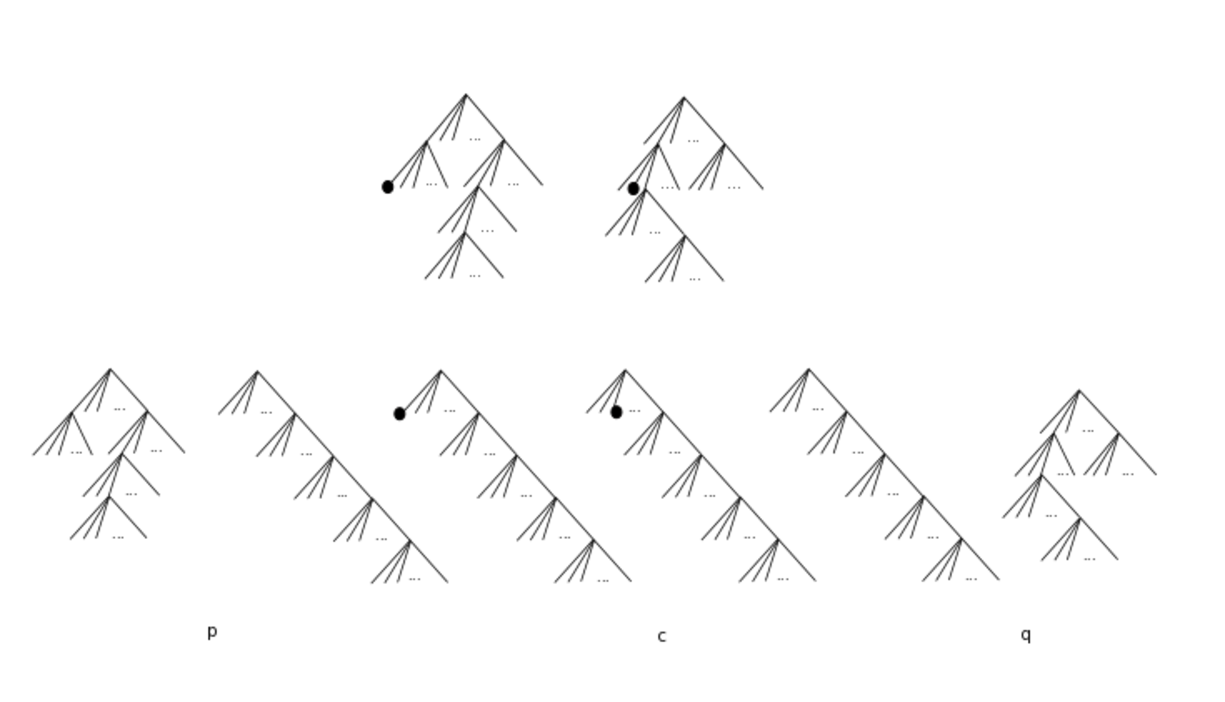}}
\caption{$\textbf{pcq}$ factorisation}
\label{fig10.1}
\end{figure}

For the estimation, 
we first consider the word length of a word only consisting of  $c_k$ with respect to 
the finite generating set  $\Sigma = \{x_0, x_1, \cdots x_{n-1}, c_0\}$.  
We let  $|w|_{\Sigma}$  denote the word length of  $w \in T_n$  
with respect to  $\Sigma$.
%then, analogously,
%the word length of $\pi_{k(n-1)+1,\ast}$ with respect to $\Sigma_{V_n}.$

\begin{proposition}\label{prop351} 
If  \, $0 < \ell < k(n-1)+ 1 = {\rm ord} \, c_{k-1}$, 
then  $|c_{k}^{\ell}|_{\Sigma} < 3k+n$. 
\end{proposition}

\begin{proof}
By the division theorem,  
$\ell -1$  can be presented uniquely as  $q(n-1) + r$  with the remainder  $0 \leq  r < n-1$,  
and the quotient $0 \leq q \leq k$  by the assumption for  $\ell$.  
We use the identities in Lemma \ref{lm1} to reduce  $c_k$  to  the expression in  
$c_0$  and  $x_i$'s.    
The computation below contains a turning point, to use either the first identity or the second one,
when  $\ell$  becomes greater than the order of the next finite order element.  
The following manipulation clarifies how we apply the identities in Lemma \ref{lm1}.  
\begin{align*} 
	c_k^{\ell} 
	& = c_{k-1}^{\ell} x_{k(n-1) - \ell} \\ 
	%& \hspace{1cm} \vdots \\ 
	& = c_q^{\ell} x_{(q+1)(n-1) -\ell} \, \cdots \, x_{k(n-1) - \ell} \\ 
	& = (c_q^{-1})^{({\rm ord} \, c_q-\ell)} x_{(q+1)(n-1) -\ell} \, \cdots \, x_{k(n-1) - \ell} \\ 
	& = (c_q^{-1})^r x_{(q+1)(n-1) -\ell} \, \cdots \, x_{k(n-1) - \ell} \\ 	
	& = x^{-1}_{q(n-1)-r} (c_{q-1}^{-1})^r x_{(q+1)(n-1) -\ell} \, \cdots \, x_{k(n-1) - \ell} \\  
	%& \hspace{1cm} \vdots \\ 
	& = x^{-1}_{q(n-1)-r} \, \cdots x^{-1}_{(n-1)-r} (c_0^{-1})^r x_{(q+1)(n-1) -\ell} \, \cdots \, x_{k(n-1) - \ell} 
\end{align*} 	

Then, 
to obtain the word expression in terms of letters in  $\Sigma_{T_n}$, 
we replace the term  $x_{\alpha}$  not in  $\Sigma$  by the 
relation  $x_{\alpha} = x_{0}^{-\gamma} x_{\delta}x_{0}^{\gamma}$, 
where  $\alpha = \gamma(n-1) + \delta$  such that  $\gamma \geq 0$  and  
$0 \leq \delta < n-1$.  
The resulting expression has a cancelling pair between two conjugates such as 
$x_{\alpha-(n-1)}x_{\alpha} = 
(x_{0}^{-(\gamma-1)} x_{\delta}x_{0}^{\gamma-1})(x_{0}^{-\gamma}x_{\delta}x_{0}^{\gamma}) = 
x_{0}^{-(\gamma-1)} x_{\delta}x_{0}^{-1}x_{\delta}x_{0}^{\gamma}$.   
The upper bound follows immediately.    
\end{proof} 

%%%%%%%%%%%%%%%%%%%%%%%%%%%%%%%%%%%%%%%%%%%%%%%%
%%%the case of Thompson's group V 
%The word length of the torsion elements in $V_n$ are more complicated.
%For $V$, 
%an infinite sequence of non-cyclic torsion elements in the generating set of $V$ are defined \cite{JWW} as pairs of all right binary trees with the labeling $(0, 1, \cdots, k-2, k-1, k)$ and $(0, 1, \cdots, k-1, k-2, k)$ and denoted by $\pi_k$s and with relations $\pi_1 = c_3^{-1}\pi_0c_3 $ and $\pi_k = x_0^{k-1} \pi_1 x_0^{-k+1}$. 

Now we estimate the word length of an element of  $T_n$ with respect to 
the finite generating set $\Sigma = \{x_0, x_1, \cdots x_{n-1}, c_0\}$,

\begin{theorem}\label{thm5} 
The following inequalities hold, 
\begin{equation*} 
	\frac{N_{\Sigma}(\omega)}{3} \leq | \omega |_{\Sigma} \leq 15(n-1)N_{\Sigma}(\omega) + 3n,  
\end{equation*} 
where  $N_{\Sigma}(\omega)$  denotes the number of carets of a reduced tree pair 
representing $\omega$.  
\end{theorem} 

\begin{proof} 
Since the  $N_{\Sigma}(x_{i})$  is at most  $3$  for  $x_{i} \in \Sigma$,  
\begin{equation*} 
	N_{\Sigma}(\omega) \leq 3 |\omega|_{\Sigma}. 
\end{equation*}   

To see the upper bound, 
suppose  $\omega \in T_n$  has a  $\textbf{pcq}$ factorization, 
\begin{equation*} 
	\omega = x_{i_1}^{r_1}x_{i_2}^{r_2} \cdots x_{i_y}^{r_y}c_{k}^{\ell} x_{j_z}^{-s_z} 
	\cdots x_{j_2}^{-s_2}x_{j_1}^{-s_1}, 
\end{equation*} 

Then, we define $D_n(\omega) = \sum_{u=1}^{y}r_u +\sum_{v=1}^{z} s_{v} + i_y + j_z + k +n$, 
and we first prove  $| \omega |_{\Sigma} \leq 3D_n(\omega)$.  
Since 
\begin{align*} 
	|\omega|_{\Sigma} 
	& = |x_{i_1}^{r_1} \, \cdots \, x_{i_y}^{r_y}c_{k}^{\ell} x_{j_z}^{-s_z} \cdots \, x_{j_1}^{-s_1}|_{\Sigma} \\ 
	& \leq  |x_{i_1}^{r_1} \, \cdots \, x_{i_y}^{r_y}|_n + |c_{k}^{\ell}|_{\Sigma} 
		+ |x_{j_z}^{-s_z} \, \cdots \, x_{j_1}^{-s_1}|_{\Sigma},  
\end{align*} 
replacing  $x_{i_{\alpha}}$  in the positive word by the relation  
$ x_{i_{\alpha}} = x_{0}^{-\gamma} x_{\delta_{\alpha}}x_{0}^{\gamma}$  where 
$i_{\alpha} = \gamma_{\alpha}(n-1) + \delta_{\alpha}$  such that  
$\gamma_{\alpha}$  is nonnegative and  $0 \leq \delta_{\alpha} < n-1$, 
we can estimate its word length as follows.  
Notice that  $\gamma_1 \leq \gamma_2 \leq \cdots \leq \gamma_y$  because  
$i_1 < i_2 < \cdots < i_y$.  
\begin{align*}
	|x_{i_1}^{r_1}x_{i_2}^{r_2} \cdots x_{i_y}^{r_y} |_{\Sigma}  
	& = |(x_0^{-1})^{\gamma_1}x_{\delta_1}^{r_1}x_0^{\gamma_1} 
	(x_0^{-1})^{\gamma_2}x_{\delta_2}^{r_1}x_0^{\gamma_2} 	
	\, \cdots \, (x_0^{-1})^{\gamma_y}x_{\delta_y}^{r_y}x_0^{\gamma_y} |_{\Sigma} \\
	& = |(x_0^{-1})^{\gamma_1}x_{\delta_1}^{r_1}x_0^{\gamma_1-\gamma_2}
	x_{\delta_2}^{r_2}x_0^{\gamma_2-\gamma_3} \,    
	\cdots \, x_0^{\gamma_{y-1}-\gamma_y}x_{\delta_y}^{r_y}x_0^{\gamma_y} |_{\Sigma} \\
	& \leq  r_1 + r_2  +\cdots + r_y + 2i_y.  
\end{align*} 
Similarly, 
\begin{equation*} 
	| x_{j_z}^{-s_z} \cdots x_{j_2}^{-s_2}x_{j_1}^{-s_1} |_{\Sigma} 
	\leq s_1 + s_2 + \cdots + s_z + 2j_z.
\end{equation*}  
Summing up these two with the estimate of Proposition \ref{prop351}, 
we obtain  $|\omega|_{\Sigma} \leq 3D_n(\omega)$.  

Now,  
$|\omega|_{\Sigma}$ can be found by looking at the multiplication 
of two elements as tree pairs. 
Recall that the tree pairs representing  $\textbf{p}$  and  $\textbf{q}$  in this setting 
might be not reduced.  
Thus  $N_{\Sigma}(\omega) \geq N_{\Sigma}(\bf{p})$, $N_{\Sigma}(\omega) \geq N_{\Sigma}(\bf{q})$.  
Also since we have  started with the reduced tree pair for  $\omega$,  
$N_{\Sigma}(\omega) = N_{\Sigma}(c_k^{\ell}) = k+1$. 
We would like to estimate  $N_{\Sigma}(\omega)$  from below by the average of 
some estimates for  $N_{\Sigma}(\textbf{p}), N_{\Sigma}(\textbf{q})$  and  $N_{\Sigma}(c_k^{\ell})$.  

Combining with the results in \cite[Theorem 5]{JSM01}, we have the following, 
\begin{align*} 
	N_{\Sigma}(\textbf{p})  &  \geq r_1 + r_2 +\cdots+ r_y \geq \frac{1}{n-1}(r_1 + r_2 +\cdots+ r_y), \\
	N_{\Sigma}(\textbf{p})  &  \geq \left\lfloor \frac{i_y}{n-1} \right\rfloor +1 \geq \frac{i_y}{n-1}, \\ 
	%N_{\Sigma}(\textbf{q})  &  \geq s_1 + s_2 + \cdots + s_z  \geq \frac{1}{n-1}(s_1 + s_2 + \cdots + s_z), \\ 
	Similarly, N_{\Sigma}(\textbf{q})  &  \geq \left\lfloor \frac{j_z}{n-1} \right\rfloor +1\geq \frac{j_z}{n-1}  \quad \text{and} \\ 
	N_{\Sigma}(c_k^{\ell})  &  = k+1 > \frac{k}{n-1}.  
\end{align*} 
Taking the average of these inequalities, 
we have 
\begin{equation*} 
	5 N_{\Sigma}(\omega) \geq \frac{1}{n-1} (D_n(\omega)-n) . 
\end{equation*} 
Thus  
\begin{equation*} 
	|\omega|_{\Sigma} \leq 3D_n(\omega) \leq 15(n-1)N_{\Sigma}(\omega) +3n
\end{equation*} 
and we are done.  
\end{proof}

\subsubsection{Word length of elements in $F_n \subset T_n$}

We then look at the word length of elements in $F_n \subset T_n$, and  
we obtain a similar result of \cite[Theorem 5.3]{JSMJ09}, 
by comparing the word length of elements in $F_n$ and $T_n$.

\begin{lemma}

If $\omega \in F_n$ 
%with $N_{\Sigma}(\omega) \geq |\omega|_{\Sigma_{n}} +1$
, then we have the following,

\begin{equation*} 
	\frac{|\omega|_{\Sigma_{n}}}{36(n-1)} \leq | \omega |_{\Sigma} \leq 45(n-1)^2|\omega|_{\Sigma_{n}} -1,  
\end{equation*} 

where $\Sigma_n = \{x_0, x_1 \dots x_{n-1}\}$ for $F_n$ 
and $\Sigma = \{x_0, x_1 \dots x_{n-1}, c_0\}$ for $T_n.$

\label{lem3.10}
\end{lemma}

\begin{proof}
An element in $F_n$ can be represented by a word in the \textbf{pcq} factorization 
corresponding to the unique reduced tree pair. 
As mentioned in Definition \ref{def3.7}, 
the torsion part of the \textbf{pcq} factorization of this word is identity, 
which means that the \textbf{pcq} factorization of this word the same as the normal form in $F_n$ defined in \cite[Theorem 1]{JSM01},
. Hence $N_{\Sigma_{n}}(\omega) = N_{\Sigma}(\omega)$.
By \cite[Theorem 1,5]{JSM01} and Theorem \ref{thm3.7}, we have the above inequality. 

\end{proof}

%%%

\subsection{Quasi-isometric embeddings}

With the ingredients ready we obtain quasi-isometric embeddings.

\begin{theorem}
 $F_n$ is embedded in $T_n$ without distortion.
\end{theorem}

In other words, the inclusion of $F_n$ in $T_n$ is a quasi-isometry.
This follows directly from Lemma \ref{lem3.10}.

\begin{theorem}
The embedding $\phi_{\ast} : T_n \hookrightarrow T_2$ is quasi-isometric.
\end{theorem}

\begin{proof}
As defined in Theorem \ref{thm3.7},  
$\phi_{\ast}$ is a label-preserving map from $T_n$ to $T_2$.  
We replace the $n$-carets in an $n$-ary tree pair by binary trees.  
Hence,  
we have $(n-1)N_{\Sigma}(\omega) = N_{\{x_0, x_1, c_0\}}(\phi_{\ast}(\omega))$.    
Then by Theorem \ref{thm5},  
\begin{equation*} 
	\frac{n-1}{45} \, |\omega|_{\Sigma} - \frac{(n-1)(n+3)}{15}   
	\leq |\phi_{\ast}(\omega)|_{\{x_0, x_1, c_0\}} 
	\leq 45(n-1) |\omega|_{\Sigma} +15,  
\end{equation*} 
which completes the proof.  
\end{proof}
%%%
\begin{remark}
This Theorem may also be proved using Thm \ref{thm5} and the independent work by Genevois \cite{AG}.
\end{remark}

By a similar construction we have, 

\begin{corollary}
The embedding $\phi_{\ast} : T_{l(m-1)+1} \hookrightarrow T_m$ is quasi-isometric, 
for any positive integer $m$ and $l$. 
The embedding $\psi_{\ast}: T_{p} \hookrightarrow T_{q}$ is quasi-isometric, 
when $(q-1)$ is divisible by $(p-1).$ 
\label{cor381}
\end{corollary}

%\begin{corollary}

%\end{corollary}

%%%

%%%
%%%
%%%

\bigskip

\section*{Acknowledgments}
First, I owe a great debt of gratitude to Professor Sadayoshi Kojima without whom this work would not be possible. 
I am also very thankful for Dr Collin Bleak, 
for carefully reviewing my paper and for precious comments and suggestions. 
Finally, I would like to thank my supervisor Professor Sakasai for the helpful comments and conversations.

\end{document}